\documentclass[11pt,a4paper,psamsfonts]{amsproc}

\usepackage{amsmath,amsthm}
\usepackage{amssymb,esint}
\usepackage{amscd}
\usepackage{enumerate}
\usepackage{graphicx,tikz}
\usepackage{anysize}
\marginsize{3.5cm}{3.5cm}{2.5cm}{2.5cm}

\newtheorem{theorem}{Theorem}

\newtheorem{corollary}[theorem]{Corollary}

\newtheorem{proposition}[theorem]{Proposition}

	\title{A Graph Decomposition motivated by \\the Geometry of Randomized Rounding}
\author{Stefan Steinerberger}
\address{Department of Mathematics, University of Washington, Seattle, WA 98195}
\thanks{The author is partially supported by NSF (DMS-2123224) and the Alfred P. Sloan Foundation.}
\keywords{Graph Decomposition, Randomized Rounding, MaxCut}
\date\empty

\begin{document}

\begin{abstract}
We introduce a graph decomposition which exists for all simple, connected graphs $G=(V,E)$. The decomposition $V = A \cup B \cup C$ is
such that each vertex in $A$ has more neighbors in $B$ than in $A$ and vice versa. $C$ is `balanced': each $v \in C$ has the
same number of neighbours in $A$ and $B$. These decompositions arise naturally from the behavior of an associated dynamical system (`Randomized Rounding') on
$(\mathbb{S}^1)^{|V|}$. Connections to judicious partitions and the \textsc{MaxCut} problem (in particular the Burer-Monteiro-Zhang heuristic) are being discussed.
\end{abstract}

\maketitle

\vspace{-0pt}

\section{Introduction and Statement}
\subsection{A Decomposition.} Let $G=(V, E)$ be a simple, connected graph. There exists a partition of the vertices into three sets
$ V = A \cup B \cup C$ (where $C$ is possibly empty)
with the following properties. Using $d_A(v), d_B(v), d_C(v)$ to denote the number of neighbors a vertex $v \in V$ has in $A,B,C$, respectively, 
\begin{enumerate}
\item each $v \in A$ has more neighbors in $B$ than it has neighbors in $A$ 
$$ d_B(v) \geq d_A(v) + \max\left\{1, d_C(v) \right\}$$
\item each $v \in B$ has more neighbors in $A$ than it has neighbors in $B$ 
$$ d_A(v) \geq d_B(v) + \max\left\{1, d_C(v) \right\}$$
\item  two vertices in $C$ are not connected by an edge: if $v \in C$, then $d_C(v) = 0$
\item each $v \in C$ has the same number of neighbors in $A$ and $B$: $d_A(v) = d_B(v)$
\item and, in particular,
$$\# E(A \cup B, C) + 2\# E(A,A) + 2 \# E(B,B) \leq 2\# E(A,B).$$
\end{enumerate}

\begin{center}
\begin{figure}[h!]
\begin{tikzpicture}[scale=1.7]
\filldraw (0,0) circle (0.04cm);
\filldraw (1,0) circle (0.04cm);
\filldraw (1,0.5) circle (0.04cm);
\filldraw (0,0.5) circle (0.04cm);
\filldraw (0.5,1) circle (0.04cm);
\draw [thick] (0,0) -- (1,0) -- (1,0.5) -- (0,0.5) -- (0.5,1) -- (1, 0.5) -- (0,0) -- (0,0.5) -- (1,0);
\draw [thick] (0,0) -- (0.5, 1) -- (1,0);
\draw [thick] (4,1) ellipse (0.7cm and 0.2cm);
\draw [thick] (4,0) ellipse (0.7cm and 0.2cm);
\draw [thick] (3, 0.5) ellipse (0.2cm and 0.4cm);
\filldraw (3.7,0) circle (0.04cm);
\filldraw (4.3,0) circle (0.04cm);
\filldraw (3.7,1) circle (0.04cm);
\filldraw (4.3,1) circle (0.04cm);
\filldraw (3,0.5) circle (0.04cm);
\draw[thick] (3.7,0) -- (4.3,0) -- (4.3,1) -- (3.7, 1) -- (3, 0.5) -- (3.7, 0) -- (3.7, 1) -- (4.3,0) -- (3, 0.5) -- (4.3, 1);
\draw[thick] (3.7,0) -- (4.3, 1);
\node at (5, 1) {$A$};
\node at (5, 0) {$B$};
\node at (2.5, 0.5) {$C$};
\end{tikzpicture}
\caption{The decomposition illustrated on $K_5$.}
\label{pic:k5}
\end{figure}
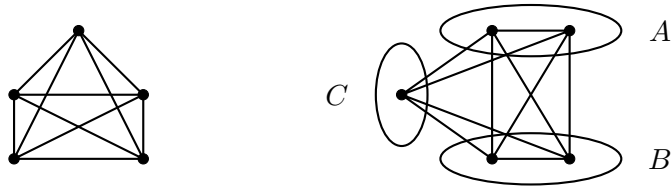
\end{center}

Consider the example of the complete graph $K_n$. Properties $(1)$ and $(2)$ require that $A$ and $B$ have the same number of vertices. Property (3) requires, for the complete graph, that $C$ has at most one element. Therefore, if $n$ is even, then $A$ and $B$ contain half the vertices and $C = \emptyset$. Conversely, if $n$ is odd, then the decomposition is given by $C$ containing a single vertex and $A, B$ containing the same amount of vertices.

\begin{center}
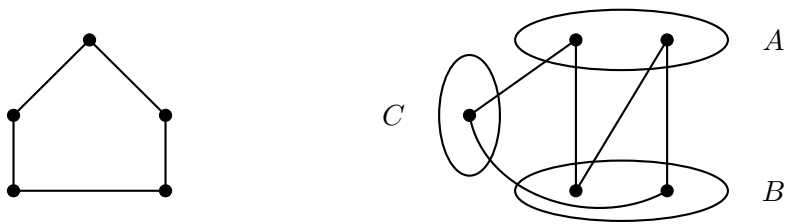
\begin{figure}[h!]
\begin{tikzpicture}[scale=2]
\filldraw (0,0) circle (0.04cm);
\filldraw (1,0) circle (0.04cm);
\filldraw (1,0.5) circle (0.04cm);
\filldraw (0,0.5) circle (0.04cm);
\filldraw (0.5,1) circle (0.04cm);
\draw [thick] (0,0) -- (1,0) -- (1,0.5) -- (0.5,1) -- (0,0.5) -- (0,0);
\draw [thick] (4,1) ellipse (0.7cm and 0.2cm);
\draw [thick] (4,0) ellipse (0.7cm and 0.2cm);
\draw [thick] (3, 0.5) ellipse (0.2cm and 0.4cm);
\filldraw (3.7,0) circle (0.04cm);
\filldraw (4.3,0) circle (0.04cm);
\filldraw (3.7,1) circle (0.04cm);
\filldraw (4.3,1) circle (0.04cm);
\filldraw (3,0.5) circle (0.04cm);
\draw[thick] (3, 0.5) -- (3.7,1) -- (3.7, 0) -- (4.3, 1) -- (4.3,0); 
\draw [thick] (4.3,0) to[out=210, in=280] (3,0.5);
\node at (5, 1) {$A$};
\node at (5, 0) {$B$};
\node at (2.5, 0.5) {$C$};
\end{tikzpicture}
\caption{The decomposition illustrated on $C_5$.}
\label{pic:halves}
\end{figure}
\end{center}

\vspace{-20pt}

It is clear that one can deduce additional properties of the decomposition if one knows more about the graph: for example, all the vertices in $C$ necessarily have even degree (because of Property (4)) and therefore $C = \emptyset$ for, say, cubic graphs (or, more generally, graphs where each vertex has an odd degree). We illustrate this on the Frucht graph and on the Petersen graph (see Fig. \ref{fig:frucht}) both of which are cubic. For these graphs, the decomposition partitions the vertices $V = A \cup B$ such that each vertex in $A$ has more neighbors in $B$ than in $A$ and vice versa. This is illustrated in Fig. \ref{fig:frucht}, the coloring of the vertices identifies the partition. Each blue vertex has more red neighbors than blue neighbors and vice versa.

\begin{center}
\begin{figure}[h!]
\begin{tikzpicture}[scale=2]
\node at (1,0.5) {\includegraphics[width=0.3\textwidth]{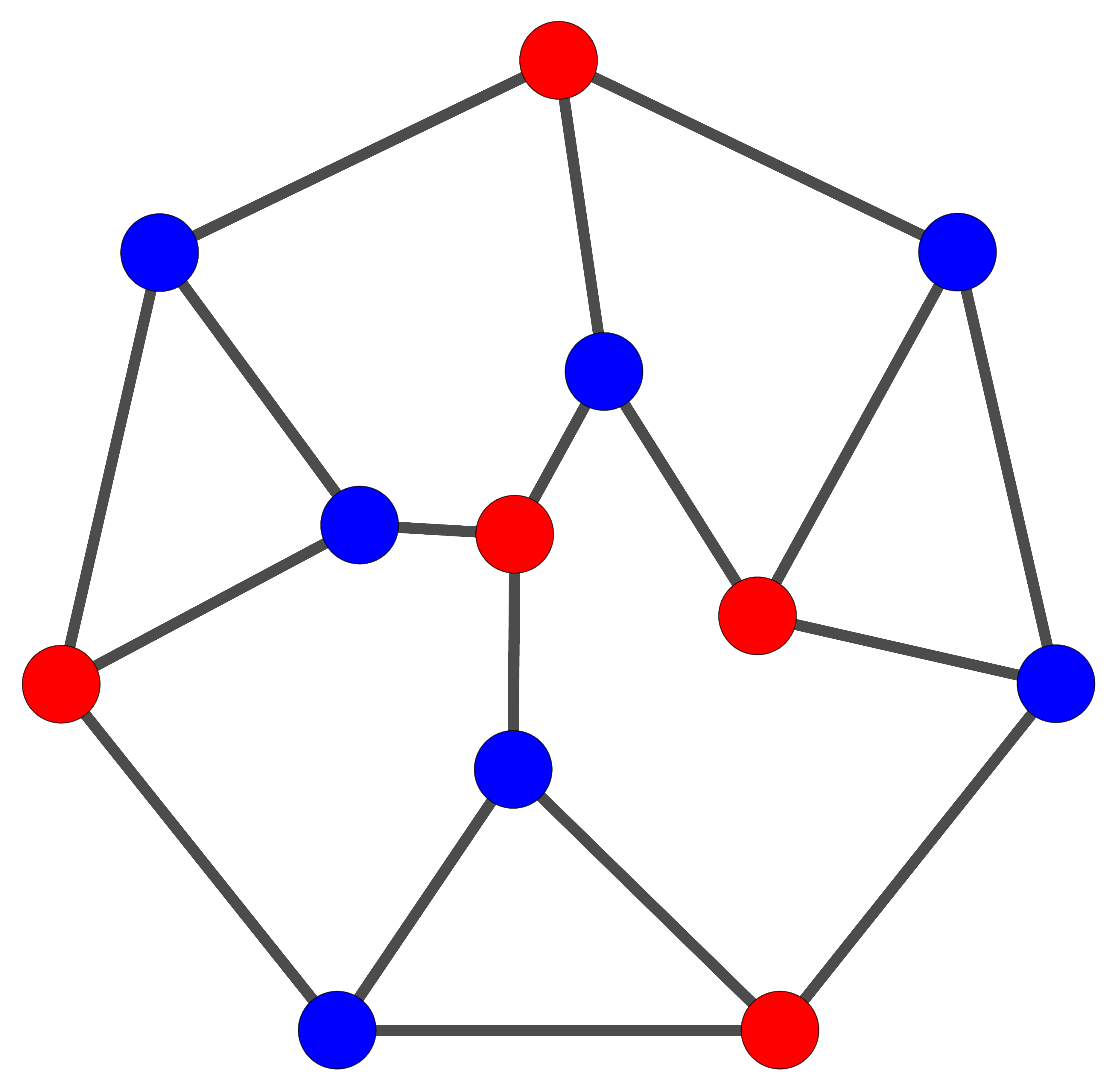}};
\node at (4,0.5) {\includegraphics[width=0.3\textwidth]{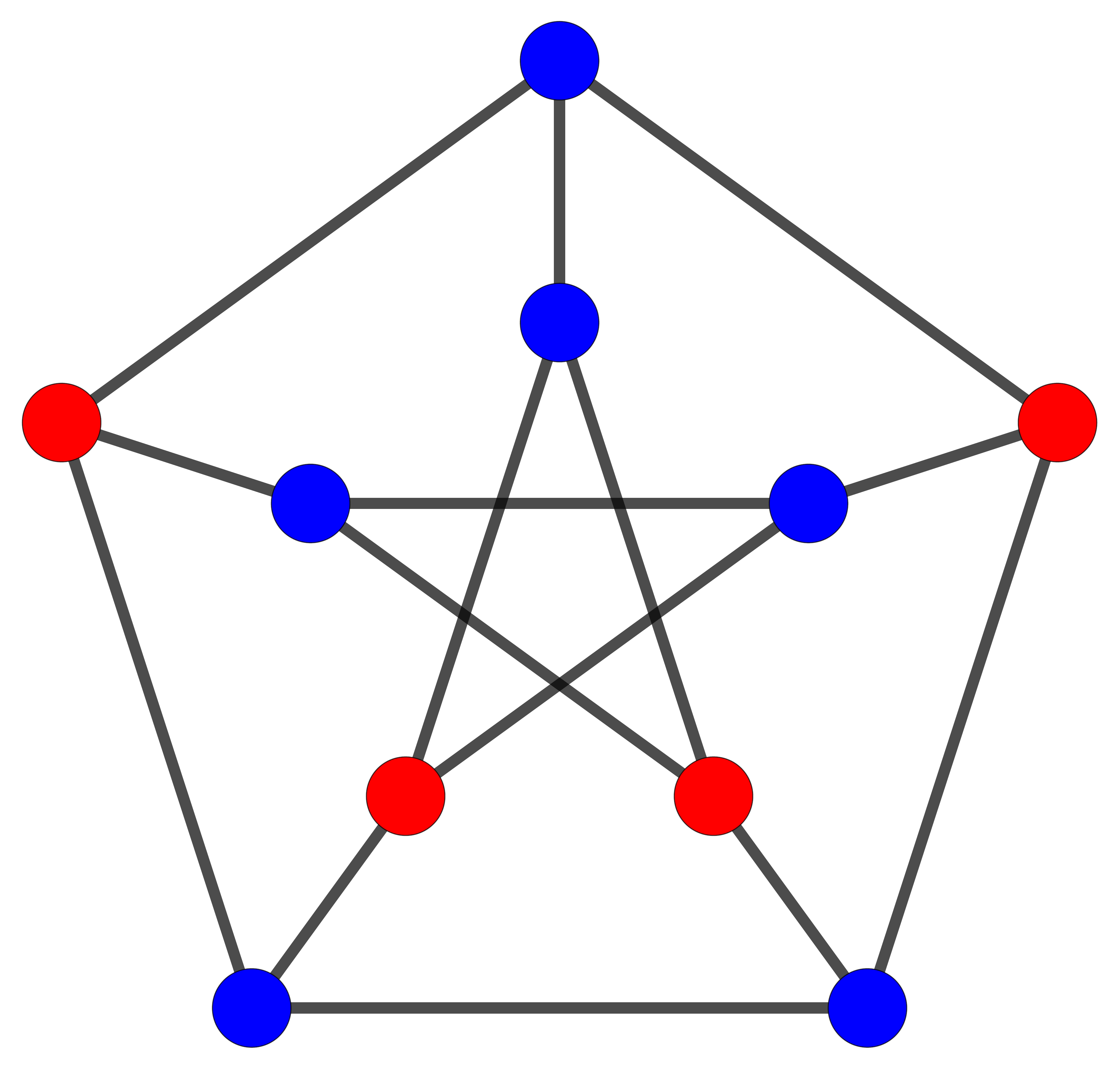}};
\end{tikzpicture}
\caption{Frucht graph and Petersen graph decomposed. }
\label{fig:frucht}
\end{figure}
\end{center}

Property (5) is, in particular, implied by Property (1) and Property (2). Let $\mathcal{A} \subseteq A$ be an arbitrary subset of $A$. Property (1) immediately shows
\begin{align*}
\# E(\mathcal{A}, C) &=  \sum_{v \in \mathcal{A}} d_C(v) \leq  \sum_{v \in \mathcal{A}} \max\left\{1, d_C(v) \right\} \leq \sum_{v \in \mathcal{A}} (d_B(v) - d_A(v)) \\
&= \# E(\mathcal{A}, B) - \sum_{v \in \mathcal{A}} d_A(v) =\# E(\mathcal{A}, B) - \left( \# E(\mathcal{A}, A \setminus \mathcal{A}) + 2\# E (\mathcal{A}, \mathcal{A}) \right) \\
&= \# E(\mathcal{A}, B) - \left( \# E(\mathcal{A}, A) + \# E (\mathcal{A}, \mathcal{A}) \right) 
\end{align*}
and therefore we the inequality
$$ \# E(\mathcal{A}, C) + \# E(\mathcal{A}, A) + \# E (\mathcal{A}, \mathcal{A}) \leq \# E(\mathcal{A}, B).$$
Property (2) then implies the same result for $\mathcal{B} \subseteq B$
$$ \# E(\mathcal{B}, C) + \# E(\mathcal{B}, B) + \# E (\mathcal{B}, \mathcal{B}) \leq \# E(\mathcal{B}, A)$$
 and by setting $\mathcal{A} = A$, $\mathcal{B} = B$ and then summing both identities, we obtain
 $$\# E(A \cup B, C) + 2\# E(A,A) + 2 \# E(B,B) \leq 2\# E(A,B).$$

\begin{center}
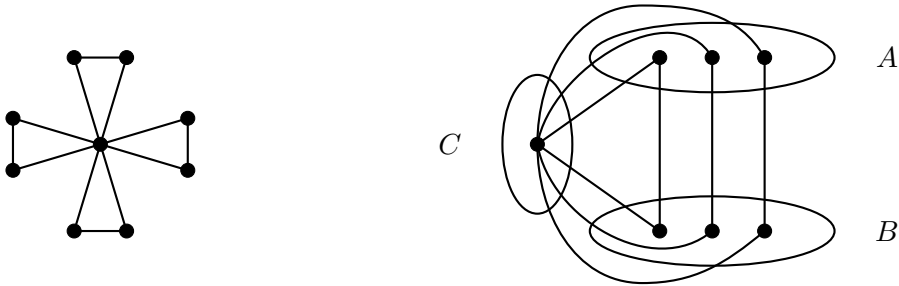
\begin{figure}[h!]
\begin{tikzpicture}[scale=2.3]
\filldraw (0.5,0.5) circle (0.04cm);
\filldraw (1,0.5+0.15) circle (0.04cm);
\filldraw (1,0.5-0.15) circle (0.04cm);
\draw [thick] (0.5,0.5) -- (1,0.5-0.15) -- (1,0.5+0.15) -- (0.5,0.5);
\filldraw (0,0.5-0.15) circle (0.04cm);
\filldraw (0,0.5+0.15) circle (0.04cm);
\draw [thick] (0.5,0.5) -- (0,0.5-0.15) -- (0,0.5+0.15) -- (0.5,0.5);
\filldraw (0.5+0.15,1) circle (0.04cm);
\filldraw (0.5-0.15,1) circle (0.04cm);
\draw [thick] (0.5,0.5) -- (0.5+0.15,1) -- (0.5-0.15,1) -- (0.5,0.5);
\filldraw (0.5+0.15,0) circle (0.04cm);
\filldraw (0.5-0.15,0) circle (0.04cm);
\draw [thick] (0.5,0.5) -- (0.5+0.15,0) -- (0.5-0.15,0) -- (0.5,0.5);
\draw [thick] (4,1) ellipse (0.7cm and 0.2cm);
\draw [thick] (4,0) ellipse (0.7cm and 0.2cm);
\draw [thick] (3, 0.5) ellipse (0.2cm and 0.4cm);
\filldraw (3.7,0) circle (0.04cm);
\filldraw (4,0) circle (0.04cm);
\filldraw (4.3,0) circle (0.04cm);
\filldraw (3.7,1) circle (0.04cm);
\filldraw (4,1) circle (0.04cm);
\filldraw (4.3,1) circle (0.04cm);
\filldraw (3,0.5) circle (0.04cm);
\draw [thick] (3.7, 0) -- (3.7, 1);
\draw [thick] (4, 0) -- (4, 1);
\draw [thick] (4.3, 0) -- (4.3, 1);
\draw [thick] (4.3,0) to[out=220, in=0] (3.6, -0.3) to[out=180, in=270] (3,0.5);
\draw [thick] (4,0) to[out=220, in=280] (3,0.5);
\draw [thick] (4,1) to[out=120, in=80] (3,0.5);
\draw [thick] (4.3,1) to[out=120, in=0] (3.6, 1.3) to[out=180, in=90] (3,0.5);
\draw [thick] (3.7,0) -- (3,0.5);
\draw [thick] (3.7,1) -- (3,0.5);
\node at (5, 1) {$A$};
\node at (5, 0) {$B$};
\node at (2.5, 0.5) {$C$};
\end{tikzpicture}
\caption{The decomposition illustrated on a friendship graph.}
\label{pic:friendship}
\end{figure}
\end{center}

We illustrate these inequalities on the friendship graph: the friendship graph is a union of triangles that all have one common vertex. Let us consider the friendship graph $G$ given by
$2n+1$ vertices and comprised of $n$ different triangles. We can decompose the graph in this case by setting $C$ to be the set containing only the central vertex shared by all
triangles and $A$ and $B$ as one would expect (see Fig. \ref{pic:friendship} for the case $n=3$). For two such subsets $\mathcal{A}, \mathcal{B}$, the expression simplifies to
\begin{align*}
\# E(\mathcal{A}, C) + \# E(\mathcal{A}, A) + \# E (\mathcal{A}, \mathcal{A}) = |\mathcal{A}| = \# E (\mathcal{A}, B) 
 \end{align*}
 The same property is satisfied for all $\mathcal{B} \subseteq B$ and
 \begin{align*}
\# E(\mathcal{B}, C) + \# E(\mathcal{B}, B) + \# E (\mathcal{B}, \mathcal{B}) = |\mathcal{B}| = \# E (\mathcal{B}, A) 
 \end{align*}
 We observe that $\# E(A \cup B, C) = 2 \# E(A,B)$
 which shows that the friendship graph is extremal with respect to the ratio of the number of edges between $A \cup B$ and $C$ when compared
 to the number of edges between $A$ and $B$.

\subsection{The main result.} We can now state the main result.

\begin{theorem}
Every simple, connected graph $G$ has a decomposition $V=A \cup B \cup C$ satisfying properties (1)-(5). Moreover,
there exists an explicit notion of energy $f_G:(\mathbb{S}^1)^{|V|} \rightarrow \mathbb{R}$ (`Randomized Rounding') such that the gradient
flow naturally leads to such a decomposition.
\end{theorem}

The existence of such a decomposition follows relatively quickly from an arbitrary \textsc{Max-Cut} solution. We show that they arise naturally
from a dynamical system that is related to heuristic approaches to \textsc{MaxCut}. The proof is constructive and
provides an explicit algorithm how to find such a set. The algorithm runs in polynomial time and has a `dynamical' flavor since gradient descent
play a role: a typical run might be much faster than the worst-case behavior. This mirrors a certain unexplained efficiency in certain
type of heuristics used for the \textsc{MaxCut} problem which is discussed in \S 2. 

\begin{corollary}
Let $G=(V,E)$ be a simple, connected graph where each vertex has odd degree. There exists a partition $V =A \cup B$ such that each vertex in $A$
has more neighbors in $B$ than in $A$ and each vertex in $B$ has more neighbors in $A$ than in $B$.
\end{corollary}
\begin{proof}
If each degree is odd, then $C = \emptyset$ and this follows immediately from the main result. However, the statement is much easier and there is a simple folklore proof.  Start with an arbitrary partition $V = A \cup B$. Each single vertex in
either set either has more neighbors in the other set (which is `good') or more neighbors in its own set (which is `bad') -- since each vertex has odd degree, a tie is not possible and each vertex is either `good' or `bad'. 
Every time we take a bad vertex and move it to the other set, the quantity $\# E(A,B)$, the number of edges between the sets $A$ and $B$ goes up by at least 1. This
means that after at most $|E|$ steps there cannot be any bad vertices left on either side of the partition.
\end{proof}

This is connected to the notion of an \textit{external partition}: an external partition of a graph is a splitting $V = A \cup B$ such that each each vertex $a \in A$ has at least as many neighbors in $B$ as it does in $A$ and vice versa. Any solution of the \textsc{MaxCut} problem is an external partition, so they always exist. Our decomposition guarantees something stronger in the sense that there are strictly more neighbors in the other set (and the price we have to pay is the existence of the third set $C$). This is part of a larger circle of questions introduced by Gerber \& Kobler \cite{gerber, gerber2}, we also refer to Ban \& Linial \cite{ban} and Bazgan, Tuza \& Vanderpooten \cite{bazgan1, bazgan2, bazgan3}. These questions are mainly concerned with partitions into two parts whereas our
decomposition will generally have three parts (with $C$ playing a slightly `asymmetric' role). 
\begin{center}
\begin{figure}[h!]
\begin{tikzpicture}[scale=2]
\node at (1,0.5) {\includegraphics[width=0.4\textwidth]{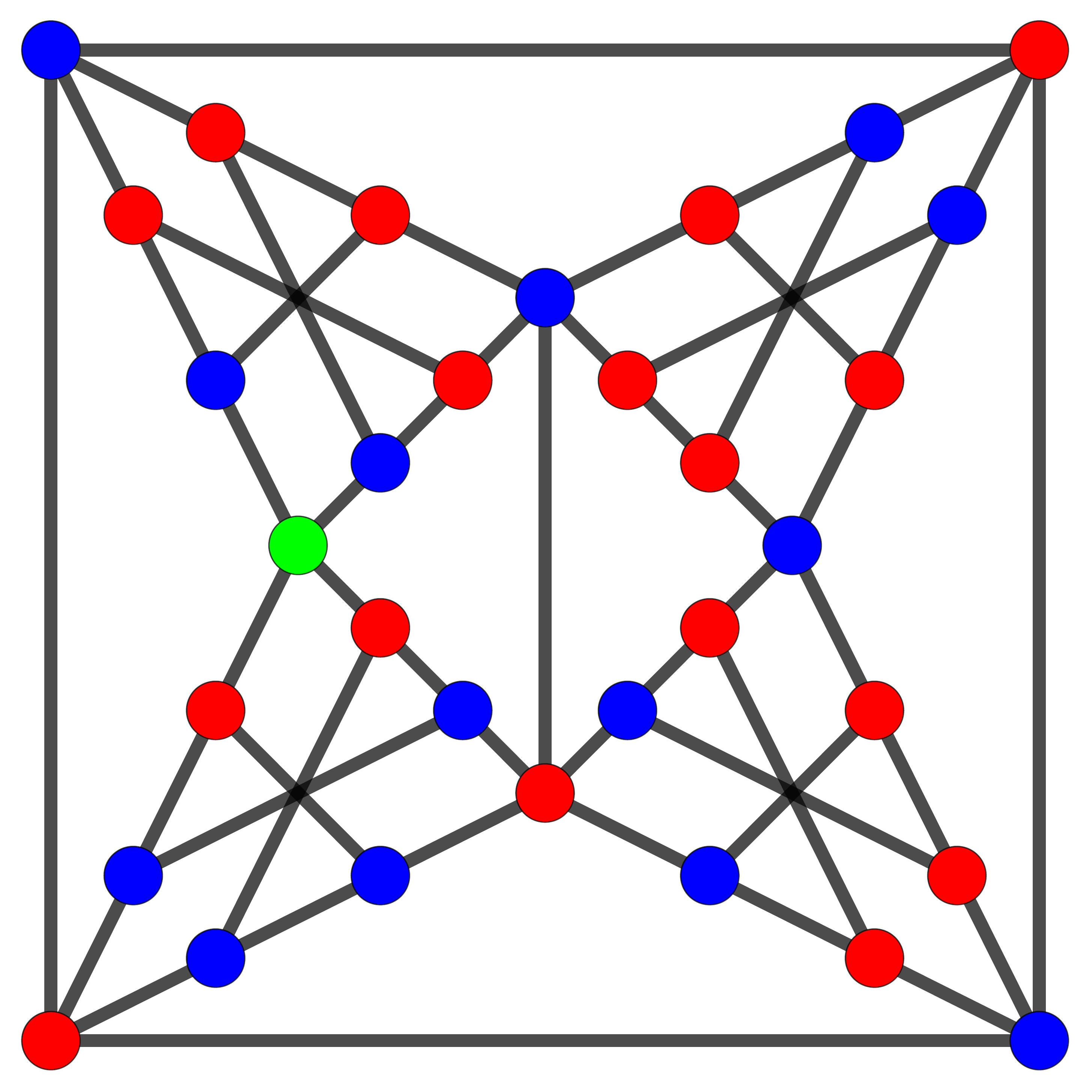}};
\end{tikzpicture}
\caption{The Thomassen graph on 32 vertices decomposed: green denotes the set $C$, the sets $A$ and $B$ are red and blue.}
\label{fig:frucht}
\end{figure}
\end{center}
\vspace{-20pt}

\subsection{Judicious Partitions.} 
We note that the decomposition immediately implies that we can partition the graph into two sets, for example $A \cup C$ and $B$, such that many edges run between them. Due to properties (3) and (4) of the decomposition, for the purpose of the subsequent argument it does not really matter how we distribute $C$, we could move some of its vertices into $A$ and the others into $B$.
\begin{corollary}
Let $G$ be a simple, connected graph with maximum degree $\Delta$ and let $V = A \cup B \cup C$ be the decomposition as above. Then the number of edges that run between $A \cup C$ and $B$ satisfies
$$ \# E(A \cup C, B) \geq \left(\frac{1}{2} + \frac{1}{3\Delta} \right) |E|.$$
\end{corollary}

This is naturally related to the concept of `judicious partitions'. We recall a result of Bollobas \& Scott \cite{bollobas} which (among other things) implies the following: if $\Delta \geq 3$ is an odd integer, then every graph with maximal degree at most $\Delta$ has a partition $V = A \cup B$ with
$$ \# E(A, B) \geq \left(\frac{1}{2} + \frac{1}{2\Delta} \right) |E|.$$
This is optimal for the complete graph. We also refer to Alon \cite{alon1}, Alon \& Halperin \cite{alon2}, Alon, Bollobas, Krivelevich \& Sudakov \cite{alon3}, Alon, Krivelevich \& Sudakov \cite{alon4}, Bollobas \& Scott \cite{bollobas2},  Edwards \cite{edwards, edwards2}, Lee, Loh \& Sudakov \cite{lee},  Shearer \cite{shearer}, Xu, Yan \& Yu \cite{xu} and references therein.

\section{Randomized Rounding}

\subsection{\textsc{MaxCut}} Our construction has an interesting connection to the \textsc{MaxCut} problem which asks for a partition $V = A \cup B$ such the number of edges that run between $A$ and $B$ is maximized
$$ \textsc{MaxCut}(G) = \max_{A,B \subset V \atop A \cap B = \emptyset} \# E(A,B).$$
The problem is known to be NP-hard, even  approximating \textsc{Max-Cut} by any factor better than $16/17 \sim 0.941$ is NP-hard \cite{bell, hastad, trev}. 
We recall the simplest randomized algorithm for \textsc{MaxCut}: by putting each vertex randomly into one of the two sets and taking the expectation, we see that
$ \textsc{MaxCut}(G) \geq 0.5 \cdot |E|.$
No algorithm improving on this elementary estimate was known until Goemans \& Williamson \cite{goemans} introduced their seminal algorithm leading to a 
$0.878 \cdot  \textsc{MaxCut}(G)$ approximation. If the Unique Games Conjecture \cite{khot} is true, this is the best possible approximation ratio for \textsc{Max-Cut} that can be computed in polynomial time.
Trevisan \cite{trev2} introduced an algorithm based on spectral graph theory which improves on the $0.5-$approximation (the constant was later improved by Soto \cite{soto}). 

\begin{center}
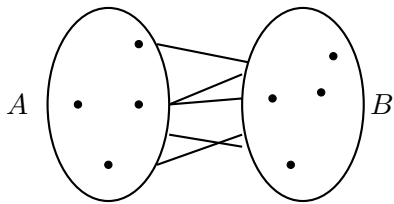
\begin{figure}[h!]
\begin{tikzpicture}[scale=0.8]
\filldraw (0,0) circle (0.06cm);
\filldraw (-0.5,1) circle (0.06cm);
\filldraw (0.5,2) circle (0.06cm);
\filldraw (0.5,1) circle (0.06cm);
\filldraw (3,0) circle (0.06cm);
\filldraw (3.2-0.5,1.1) circle (0.06cm);
\filldraw (3.7,1.8) circle (0.06cm);
\filldraw (3.5,1.2) circle (0.06cm);
\draw [thick] (0,1) ellipse (1cm and 1.6cm);
\draw [thick] (3.2,1) ellipse (1cm and 1.6cm);
\draw [thick] (0.8,2) -- (2.3,1.7);
\draw [thick] (1,1) -- (2.2,1.5);
\draw [thick] (0.8,0) -- (2.2,0.5);
\draw [thick] (1,0.5) -- (2.2,0.3);
\draw [thick] (1,1) -- (2.2,1.1);
\node at (-1.5, 1) {$A$};
\node at (4.5, 1) {$B$};
\end{tikzpicture}
\caption{$\textsc{MaxCut}$: partitioning vertices of a Graph into two sets $V = A \cup B$ so that many edges run between them.}
\end{figure}
\end{center}

\vspace{-10pt}

\begin{corollary} Let $G$ be a graph with maximal degree $\Delta$. The decomposition gives rise to a cut of size
$$  \# E(A \cup C, B) \geq  \left(\frac{1}{2} + \frac{1}{3\Delta} \right) \textsc{MaxCut}(G).$$
\end{corollary}

This is an easy consequence of Corollary 3 since $\textsc{MaxCut}(G) \leq |E|$. We note that the problem of finding approximations
of \textsc{MaxCut} is easier when one assumes that the graph has bounded degree. Stronger results can, for example, be found in 
\cite{bazgan, bondy, halperin} when $\Delta=3$. There is an inverse direction: given a solution to the \textsc{MaxCut}
problem, we can quickly derive a valid decomposition from it. This naturally raises the question of how quickly one can compute our decomposition
for a given graph: our proof is constructive and runs in polynomial time but may not be the fastest way of computing it. Another interesting
question is whether a typical decomposition discovered by running gradient descent on our functional might actually have the property that $ \# E(A \cup C, B)$
is often close to $\textsc{MaxCut}(G)$, i.e. whether this decomposition is a good \textsc{MaxCut} heuristic (see also \cite{stein}).

\subsection{Randomized Rounding}
Goemans \& Williamson \cite{goemans} propose to relax 
$$ 2 \cdot |E| - 4 \cdot \textsc{MaxCut}(G) =  \min_{x_i \in \left\{-1,1\right\}} \sum_{i,j=1}^{n} a_{ij} x_i x_j$$
by replacing the $x_i \in \left\{-1,1\right\}$ with unit vectors $v_i \in \mathbb{R}^n$ and $x_i x_j$ with $\left\langle v_i, v_j\right\rangle$. This is a more general problem
but one that can be solved with semi-definite programming in polynomial time. The optimal set of vectors is usually not going to be contained in two antipodal points:
Goemans \& Williamson propose to use a random hyperplane to induce a partition. This step is known as Randomized Rounding (see Fig. \ref{fig:rounding}).

 \begin{center}
\begin{figure}[h!]
\begin{tikzpicture}[scale=0.9]
\draw [thick] (0,0) circle (2cm);
\filldraw (0,0) circle (0.04cm);
\filldraw (2,0) circle (0.06cm);
\filldraw (1.9,0.6) circle (0.06cm);
\filldraw (1.92,-0.5) circle (0.06cm);
\filldraw (-1.82,0.8) circle (0.06cm);
\filldraw (-1.92,-0.5) circle (0.06cm);
\filldraw (-2,0) circle (0.06cm);
\filldraw (1.4142, 1.4142) circle (0.06cm);
\node at (2.3, 0) {$\theta_4$};
\node at (2.2, -0.6) {$\theta_6$};
\node at (-2.2, -0.6) {$\theta_{2}$};
\node at (-2.3, 0) {$\theta_{5}$};
\node at (1.75, 1.41) {$\theta_{7}$};
\node at (2.3, 0.55) {$\theta_{3}$};
\node at (-2.2, 0.9) {$\theta_{1}$};
\draw[dashed] (-1.8, -2.2) -- (1.8, 2.2);
\draw (-2,0) ellipse (1cm and 1.5cm);
\draw (1.8,0.2) ellipse (1.2cm and 1.5cm);
\end{tikzpicture}
\caption{Randomized Rounding: for given $\left\{\theta_1, \dots, \theta_n\right\} \subset \mathbb{S}^1$, we can pick a random line through the origin and the partition the vertices of the graph according to the two half-spaces.}
\label{fig:rounding}
\end{figure}
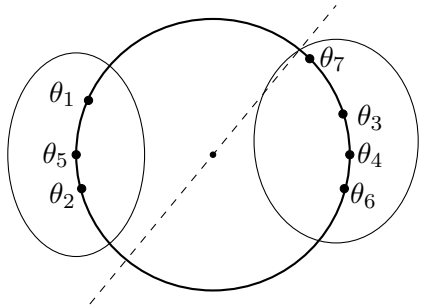
\end{center}

 Burer, Monteiro \& Zhang \cite{burer3} suggest that it might be possible to bypass the (computationally expensive) semi-definite programming by working directly with the relaxed problem in $\mathbb{R}^2$. Parametrizing points on $\mathbb{S}^1$ by an angle $\theta \in \mathbb{S}^1$, we see that
$$ \left\langle v_{\theta_i}, v_{\theta_j} \right\rangle = \cos{(\theta_i - \theta_j)}$$
and this leads to the notion of energy $f:(\mathbb{S}^1)^n \cong [0,2\pi]^n \rightarrow \mathbb{R}$
 $$ f(\theta_1, \dots, \theta_n) = \sum_{i,j=1}^{n} a_{ij} \cos{(\theta_i - \theta_j)},$$
 where $n = |V|$ and $A = (a_{ij})_{i,j=1}^{n}$ is the adjacency matrix of the graph. In particular, this suggests that we should minimize $f$ via some optimization scheme (simple gradient descent seems to suffice in practice) and then partition the arising set of vertices via randomized rounding.  
 The Burer-Monteiro-Zhang \cite{burer3} approach does not come with guarantees, however, in practice, this method works amazingly well. In a 2018 paper, Dunning, Gupta \& Silberholz \cite{dunning} compared 37 different heuristics over 3296 problem instances concluding that the Burer-Monteiro-Zhang approach was among the most effective. It is not theoretically understood why this relaxation works so well and this is of great interest, we refer to  Boumal, Voroninski, Bandeira \cite{boumal, boumal2}, Ling \cite{ling0} and Ling, Xu \& Bandeira \cite{ling}. It is also related to questions regarding Kuramoto oscillators \cite{burer1, burer2, burer3, review1, review2,kuramoto, jianfeng, taylor, townsend} and recent hardware-based oscillator approaches to \textsc{MaxCut}, see \cite{chou, mallick, wang, wang2}.

\subsection{The Functional.} It was recently pointed out by the author \cite{stein} that if we relax things to $\mathbb{S}^1$, we should aim to maximize efficiency with respect to the randomized rounding step. We restrict ourselves to energy functionals of the type
 $$ f(\theta_1, \dots, \theta_n) = \sum_{i,j=1}^{n} a_{ij} \cdot g(\theta_i - \theta_j),$$
where $g: \mathbb{S}^1 \cong [0,2\pi] \rightarrow \mathbb{R}$ is assumed to 
\begin{enumerate}
\item be differentiable everywhere,
\item be symmetric in the sense of $g(x) = g(-x)$ and
\item to assume its maximum in $g(0) = 1$ and its minimum in $g(\pi) = -1$.
\end{enumerate}
 Then it was shown in \cite{stein} that if we are given a minimal energy configuration of such an energy, then the expected number of edges recovered by a randomized rounding partition satisfies
$$ \mathbb{E}_{} ~\mbox{edges} \geq \left(  \min_{0 \leq x \leq \pi}   \frac{2}{\pi}  \frac{x}{1 -g(x)}  \right) \cdot \emph{\textsc{MaxCut}}(G).$$
For suitable choices of $g$, the ratio in front can be arbitrarily close to 1 (and is even equal to 1 in one special case, see below). Naturally, this shifts the problem of finding a good \textsc{MaxCut} to the question of how to find $(\theta_1, \dots, \theta_n) \in (\mathbb{S}^1)^n$ for which such a functional is small. Certainly one cannot hope to find a global minimum but perhaps it is possible to find a configuration which has small energy. It turns out that this is sufficient: finding values close to the global minimum implies finding cuts close to the maximal cut since for any $(\theta_1, \dots, \theta_n) \in (\mathbb{S}^1)^n$
$$ \mathbb{E}_{} ~\mbox{edges}(\theta_1, \dots, \theta_n) \geq  \left[\left(  \min_{0 \leq x \leq \pi}   \frac{2}{\pi}  \frac{x}{1 -g(x)}  \right)  \right] \cdot \left(\frac{|E|}{2} - \frac{1}{4} f(\theta_1, \dots, \theta_n)\right)$$
It was shown empirically in \cite{stein} that minimizing such functionals leads to very good performance: the underlying energy landscape
seems to be benign but there is relatively little that has been made precise in that direction.
We note that the approach in \cite{stein} singles out one particular function $g: \mathbb{S}^1 \rightarrow \mathbb{R}$
$$ g(x) = 1 - \frac{2}{\pi} \cdot d_{\mathbb{S}^1}(0, x) = \begin{cases} 1 - \frac{2 x}{\pi} \qquad &\mbox{if}~0 \leq x \leq \pi \\
 1 - \frac{2 (2\pi -x)}{\pi} \qquad &\mbox{if}~\pi \leq x \leq 2\pi, \end{cases}$$
which is also shown in Fig. \ref{fig:bestfun}.
 This is the function for which the ratio in front is equal to 1, the global minimum
of the functional corresponds to a configuration from which \textsc{MaxCut} can be recovered. 
\begin{proposition}[see also \cite{stein}] We have, for this particular choice of $g:\mathbb{S}^1 \rightarrow \mathbb{R}$,
$$2\cdot |E| - 4 \cdot \mathbb{E}\left(\emph{edges cut by randomized rounding}\right)=  \sum_{i,j=1}^{n} a_{ij} \cdot g(\theta_i - \theta_j).$$
\end{proposition}

 \begin{center}
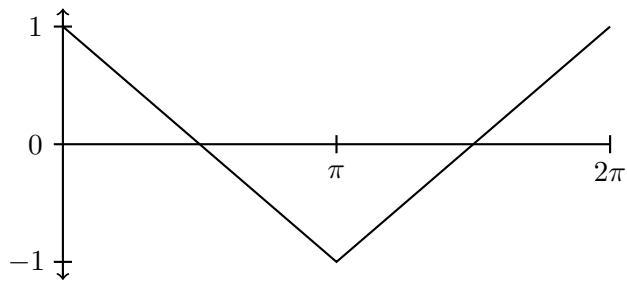
\begin{figure}[h!]
\begin{tikzpicture}[scale=1.2]
\draw [thick, <->] (0,-1.5) -- (0,1.5);
\node at (-0.3, 1.3) {$1$};
\draw [thick] (-0.1, 1.3) -- (0.1, 1.3);
\node at (-0.4, -1.3) {$-1$};
\draw [thick] (-0.1, -1.3) -- (0.1, -1.3);
\node at (-0.3, 0) {$0$};
\draw [thick] (-0.1, 0) -- (6, 0);
\draw [thick] (3, -0.1) -- (3,0.1);
\draw [thick] (6, -0.1) -- (6,0.1);
\node at (3,-0.3) {$\pi$};
\node at (6,-0.3) {$2\pi$};
\draw [thick] (0,1.3) -- (3,-1.3) -- (6,1.3);
\end{tikzpicture}
\caption{The function $g:\mathbb{S}^1 \rightarrow \mathbb{R}$.}
\label{fig:bestfun}
\end{figure}
\end{center}

This is merely a computation. Minimizing the functional is thus equivalent to maximizing the number of edges that are being cut by randomized rounding. It raises the  question whether one should perhaps not simply optimize $(\theta_1, \dots, \theta_n) \in (\mathbb{S}^1)^n$
using this particular functional. From a practical point of view, this is not currently understood but will also not be relevant for what follows.
For the remainder of the paper, we will only consider the particular functional 
$$ f(\theta_1, \dots, \theta_n) = \sum_{i,j=1}^{n} a_{ij} \cdot g(\theta_i - \theta_j),$$
where $g$ is exactly the function that makes Proposition 5 work
$$ g(x) = 1 - \frac{2}{\pi} \cdot d_{\mathbb{S}^1}(0, x) = \begin{cases} 1 - \frac{2 x}{\pi} \qquad &\mbox{if}~0 \leq x \leq \pi \\
 1 - \frac{2 (2\pi -x)}{\pi} \qquad &\mbox{if}~\pi \leq x \leq 2\pi. \end{cases}$$
We refer to $f(\theta_1, \dots, \theta_n)$ as the `energy' of the configuration $(\theta_1, \dots, \theta_n) \in (\mathbb{S}^1)^n$
and note the following elementary properties:
\begin{enumerate}
\item $f$ is continuous
\item $f$ is piecewise affine
\item $f$ is bounded from below
$$ \min_{\theta \in (\mathbb{S}^1)^n} f(\theta_1, \dots, \theta_n) = 2\cdot |E| - 4 \cdot \textsc{MaxCut}(G) \geq -2|E|$$
\item for each $\theta \in (\mathbb{S}^1)^n$, the function $h:\mathbb{S}^1 \rightarrow \mathbb{R}$ given by
$$ h(x) = f(\theta_1, \dots, \theta_{i-1}, x, \theta_{i+1}, \dots, \theta_n)$$
is piecewise linear and its derivative at each point where it is defined is of the form $2k/\pi$ for some $k \in \mathbb{Z}$.\\
\end{enumerate}

The first two properties are obvious, for the third property we note that Proposition 5 shows that the minimum cannot be smaller. If we take a \textsc{MaxCut} partition $V = A \cup B$ and define
$ \theta_i = 0$ for all $i \in A$ and $\theta_i = \pi$ for all $i \in B$, a short computation shows that equality is attained. Alternatively, we could argue that for this particular configuration almost every line in randomized rounding does realize the \textsc{MaxCut} and the desired statement follows. The fourth property follows quickly from the definition but is important and deserves a  careful explanation (see Fig. \ref{fig:prop4}).  
Let us fix $\theta_i$. There is the antipodal point of $\theta_i$ which then naturally subdivides $\mathbb{S}^1$ into two connected regions $L$ and $R$. Let us try to understand
what happens, in the explicit setting shown in Fig. \ref{fig:prop4}, to the function
$$ h(x) = f(\theta_1, \dots, \theta_{i-1}, x, \theta_{i+1}, \dots, \theta_n)$$
when we move $x$ in a neighborhood of $\theta_i$. If we move it to the right (i.e. increase $x$), then we increase the distance to two vertices that we
are connected to by an edge (the neighbors in $L$) while decreasing the distance to one vertex that we are connected to via an edge (the neighbor in $R$): this means the function will decrease.

 \begin{center}
\begin{figure}[h!]
\begin{tikzpicture}[scale=2]
\node at (0, -0.2) {$0$};
\draw [thick] (0,-0.1) -- (0,0.1);
\draw [thick] (0, 0) -- (6, 0);
\draw [thick] (6, -0.1) -- (6,0.1);
\node at (6,-0.2) {$2\pi$};
\filldraw (2, 0) circle (0.04cm);
\node at (2, -0.2) {$\theta_i$};
\draw [very thick] (5, 0) circle (0.04cm);
\node at (5, -0.2) {\small antipodal point};
\filldraw (0.3, 0) circle (0.04cm);
\filldraw (1.2, 0) circle (0.04cm);
\filldraw (2.5, 0) circle (0.04cm);
\filldraw (3.2, 0) circle (0.04cm);
\filldraw (4.5, 0) circle (0.04cm);
\filldraw (5.5, 0) circle (0.04cm);
\draw [thick] (2,0) to[out=30, in=150] (5.5, 0);
\draw [thick] (2,0) to[out=20, in=160] (3.2, 0);
\draw [thick] (2,0) to[out=120, in=40] (1.2, 0);
\draw [ultra thick] (2, -0.1) -- (2,0.3);
\draw [ultra thick] (5, -0.1) -- (5,0.3);
\node at (0.8, 0.1) {$L$};
\node at (3.9, 0.1) {$R$};
\node at (5.8, 0.1) {$L$};
\end{tikzpicture}
\caption{Understanding the derivative with respect to a single coordinate.}
\label{fig:prop4}
\end{figure}
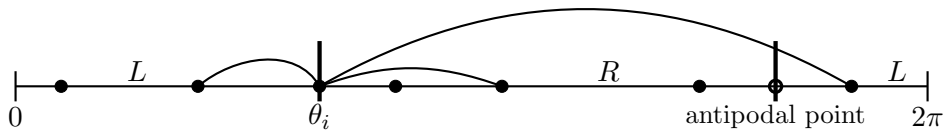
\end{center}

Conversely, going to the left (decreasing $x$), the function will increase. This holds true more generally, whenever the number of points we are connected
to `on the left' is larger than the number of points `on the right' (or vice versa), then we can further decrease the energy. Note that for the purpose of these
considerations, points that lie in the antipodal point of $\theta_i$ (should they exist) contribute to `both sides' since, no matter where we move to, the distance to them is being
decreased. \\

We conclude this section by arguing that this functional has a `benign' energy landscape. More precisely, the proof will show the following.
\begin{corollary}
For any arbitrary initial configuration $\theta \in (\mathbb{S}^1)^n$, there exists a continuous path $\gamma:[0,1] \rightarrow (\mathbb{S}^1)^n$
\begin{itemize}
\item such that $\gamma(0) = \theta$,
\item $f(\gamma(t))$ is monotonically decreasing
\item $\gamma(1)$ is a configuration from which it is possible to obtain the decomposition.
\end{itemize}
\end{corollary}
This shows that the procedure does not get trapped in `bad' local minima because they do not exist: if we end up in a strict local minimum, we can always obtain a partition.
However, the procedure may encounter critical points and even entire regions where $f$ is locally constant. In particular, $f(\gamma(t))$ may be piecewise constant along the path guaranteed by Corollary 6, it need not be strictly monotonically decreasing. It is always possible to either leave each critical point (or `region of flatness') using an explicit procedure or it is possible to turn said critical point into a way of obtaining the decomposition.

\section{Proof of the Theorem}
\subsection{Outline.} 
The strategy is quite simple: we essentially prove Corollary 6. We do this by illustrating how it is possible to continuously deform the points  (up to at most a finite number of jumps) in a way that does not increase the energy. With these procedures, we will find a continuous deformation to a critical point. A different type of procedure will then allow us to escape most critical points in a way that will further decrease the energy by at least 1. The only type of critical point that we cannot escape will be the ones corresponding to the type of decomposition whose existence we are trying to prove. Since the energy is bounded from below
$$ \min_{\theta \in (\mathbb{S}^1)^n} f(\theta_1, \dots, \theta_n) = 2\cdot |E| - 4 \cdot \textsc{MaxCut}(G) \geq -2 |E|$$
we cannot escape critical points forever and have to ultimately end up in one where we cannot escape. This then corresponds to a decomposition of the type we try to construct.
We start with a simple case-distinction: either all the $n$ variables $(\theta_1, \dots, \theta_n) \in (\mathbb{S}^1)^n$ are located at two antipodal points or not. More formally
\begin{enumerate}
\item either there exists $\theta \in \mathbb{S}^1$ such that 
$$ \forall~1 \leq i \leq n: ~\theta_i = \theta \qquad \mbox{or} \qquad \theta_i = \theta + \pi$$
\item or not.
\end{enumerate}
We will refer to the first case as an `antipodal configuration'. The first part of the argument is to show that each configuration can be continuously deformed into an antipodal
configuration without increasing the energy. We note that ultimately both steps of the argument, (1) reducing an arbitrary configuration to an antipodal configuration without increasing
the energy and (2) modifying an antipodal configuration to either result in a valid decomposition or in a configuration with lower energy, might be applied several times in 
a row: every time (2) does not result in a valid decomposition, it will result in a new non-antipodal configuration with lower energy which will then be transformed into another antipodal configuration.

\subsection{The Rotation Argument.} 
We start by assuming that the configuration $(\theta_1, \dots, \theta_n) \in (\mathbb{S}^1)^n$ is not an antipodal configuration. Then there exists at least one point $\theta_i$ such that not all points are
contained in $\left\{\theta_i, \theta_i + \pi \right\}$. 
We partition the $n$ points $\theta_1, \dots, \theta_n$ into 4 sets
$$ \left\{\theta_1, \dots, \theta_n\right\} = A_{\theta_i} \cup A_{\theta_i + \pi} \cup A_L \cup A_R.$$
 The first set $A_{\theta_i}$ all the points that are in the same spot as $\theta_i$. This set is non-empty since it contains, tautologically, $\theta_i$ but it may also contain other points.
 The set $A_{\theta_i + \pi}$ contains all the antipodal vertices, all the points that are exactly antipodal to the points in $A_{\theta_i}$. This set might be empty.
 $A_L$, $A_R$ are the points on the left and on the right side. Since everything is only defined up to rotational symmetry, it is a bit arbitrary which of the two sets is considered to be on the `left' side. For simplicity, we assume things to be labeled as in Fig. \ref{pic:rotation}. Since, by definition of $\theta_i$, not all points are
contained in $\left\{\theta_i, \theta_i + \pi \right\}$, we can conclude that $A_L \cup A_R \neq \emptyset$.
We can now consider four quantities: the number of edges between vertices in $A_{\theta_i}$ and vertices in $A_L$ and $A_R$, respectively, together with the number
of edges between vertices in $A_{\theta_i + \pi}$ (this may be the empty set) and vertices in $A_L$ and $A_R$, respectively. A short computation shows (see Fig. \ref{pic:rotation})
if 
$$ \#  E(A_{\theta_i}, A_R) + \#  E(A_{\theta_i + \pi}, A_L) > \#  E(A_{\theta_i}, A_L) + \#  E(A_{\theta_i + \pi}, A_R),$$
then a counterlockwise rotation of all the points in $A_L \cup A_R$ by the same small angle decreases the energy. This can be done at least until the first point in $A_L \cup A_R$
starts being in either $A_{\theta_i}$ or $A_{\theta_i + \pi}$ which is when we stop.

\begin{center}
\begin{figure}[h!]
\begin{tikzpicture}
\draw [thick] (0,0) circle (2cm);
\filldraw (0,-2) circle (0.08cm);
\node at (0, -2.5) {$A_{\theta_i}$};
\node at (1, 2.5) {$A_{\theta_i + \pi}$ (possibly empty)};
\draw [dashed] (0,-2) -- (0,2);
\node at (-2.3,0) {$A_L$};
\node at (2.3,0) {$A_R$};
\draw [ultra thick] (0,-2) circle (0.2cm);
\draw [ultra thick] (0,2) circle (0.2cm);
\draw [thick, <-] (-2.2, 1) to[out= 60, in = 210] (-1.2, 2);
\draw [thick, ->] (-2.2, -1) to[out= 300, in = 150] (-1.2, -2);
\draw [thick, <-] (-2.2, 1) to[out= 60, in = 210] (-1.2, 2);
\draw [thick, ->] (1.2, -2) to[out= 30, in = 240] (2.2, -1) ;
\draw [thick, <-] (1.2, 2) to[out= 330, in = 120] (2.2, 1) ;
\end{tikzpicture}
\caption{Rotating every point except the points in $A_{\theta_i}$ and $A_{\theta_i + \pi}$.}
\label{pic:rotation}
\end{figure}
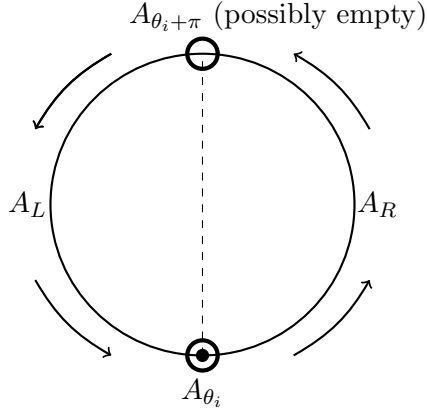
\end{center}

 Conversely, if
$$ \#  E(A_{\theta_i}, A_R) + \#  E(A_{\theta_i + \pi}, A_L) < \#  E(A_{\theta_i}, A_L) + \#  E(A_{\theta_i + \pi}, A_R),$$
then a clockwise rotation will decrease the energy -- again this will be carried out until the first point in $A_L \cup A_R$
hits either $A_{\theta_i}$ or $A_{\theta_i + \pi}$. This leads to a continuous deformation of points for which we
are guaranteed that the energy will decrease until a point that was not in $ A_{\theta_i} \cup A_{\theta_i + \pi}$ suddenly is. If
$$ \#  E(A_{\theta_i}, A_R) + \#  E(A_{\theta_i + \pi}, A_L)  =  \#  E(A_{\theta_i}, A_L) + \#  E(A_{\theta_i + \pi}, A_R),$$
then it is possible to rotate in either direction without changing the energy: the energy is not going to decrease but it is
not going to increase either. We rotate until a point that was not in $ A_{\theta_i} \cup A_{\theta_i + \pi}$ suddenly is.

We now observe that this procedure never breaks the property of `two points are in the same spot'.  If two points are
in the same spot, they will remain in the same spot in all future steps of the algorithm no matter what happens. Likewise,
the algorithm preserves the property of `two points are in antipodal position'. Once two points are in antipodal position,
they will remain so for all time.
We see that this procedure increases the number of pairs of points that are either in the same spot or antipodal to each other always
by at least 1. Thus it must end after a finite amount of steps and it can only end in an antipodal configuration.\\

\textit{A Different Perspective.} This paragraph can be skipped, its purpose is to present the rotation argument in a different form
for clarity of exposition. We fixed the sets $A_{\theta_i}$ and $A_{\theta_i + \pi}$ and then rotated all the
other points. However, we can change the perspective and instead fix all the points and merely move all the points in 
$A_{\theta_i} \cup A_{\theta_i + \pi}$ either both a little bit to the left or a little bit to the right (see Fig. \ref{fig:otherexp}).
Again, this naturally partitions the remaining space into two connected regions which we call $A_L$ and $A_R$. Suppose now we move
all the points in $A_{\theta_i} \cup A_{\theta_i + \pi}$ slightly too the left. This decreases distances for the edges $\#E (A_{\theta_i}, A_L)$
and $\#E (A_{\theta_i+\pi}, A_R)$ while increasing the distance of edges $\#E (A_{\theta_i}, A_R)$
and $\#E (A_{\theta_i+\pi}, A_L)$. If the total net change has a sign, we can go in the direction where the net change is negative until colliding with (one of the now stationary)
points for the first time (and this collision may happen in either of the two sets $A_{\theta_i}$ and $A_{\theta_i + \pi}$). If the net change is invariant, we can go in either direction until colliding. Also here, we see that more and more points end up being
in the same spot or antipodal and this relationship is never broken by subsequent applications of the algorithm.
 \begin{center}
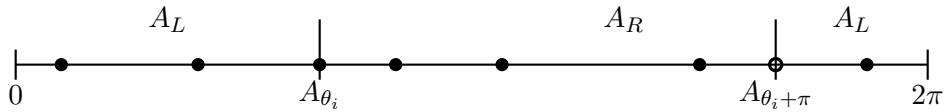
\begin{figure}[h!]
\begin{tikzpicture}[scale=2]
\node at (0, -0.2) {$0$};
\draw [thick] (0,-0.1) -- (0,0.1);
\draw [thick] (0, 0) -- (6, 0);
\draw [thick] (6, -0.1) -- (6,0.1);
\node at (6,-0.2) {$2\pi$};
\filldraw (2, 0) circle (0.04cm);
\node at (2, -0.2) {$A_{\theta_i}$};
\draw [very thick] (5, 0) circle (0.04cm);
\node at (5, -0.2) {$A_{\theta_i + \pi}$};
\filldraw (0.3, 0) circle (0.04cm);
\filldraw (1.2, 0) circle (0.04cm);
\filldraw (2.5, 0) circle (0.04cm);
\filldraw (3.2, 0) circle (0.04cm);
\filldraw (4.5, 0) circle (0.04cm);
\filldraw (5.6, 0) circle (0.04cm);
\draw [thick] (2, -0.1) -- (2,0.3);
\draw [thick] (5, -0.1) -- (5,0.3);
\node at (4, 0.3) {$A_R$};
\node at (1, 0.3) {$A_L$};
\node at (5.5, 0.3) {$A_L$};
\end{tikzpicture}
\vspace{-5pt}
\caption{Fixing all the points and moving $A_{\theta_i} \cup A_{\theta_i + \pi}$.}
\label{fig:otherexp}
\end{figure}
\end{center}

\subsection{An Antipodal Configuration.}
It remains to understand the case where all the angles are located in either the same point or the anti-podal point (see Fig. \ref{pic:type1}). 
We use $C$ to denote the mid-point between $A$ and $B$ (there are two possible choices for $C$ and it does not matter which one we pick).
If we now pick any vertex in $A$ or $B$ and move it towards $C$, the energy could further decrease: if so, then we have found a further
direction in which to decrease the energy, we use this direction to actually decrease the energy by at least 1 and then we use the rotation argument to return
to another antipodal configuration (now with lower energy). This can only happen finitely many times. Eventually, at some point, we have the property that moving any vertex from $A$ or $B$ towards $C$ does not decrease
the energy further. For each vertex, it might increase the energy or it might preserve the energy.
\begin{center}
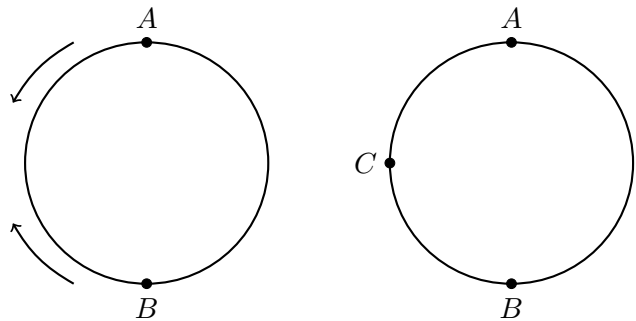
\begin{figure}[h!]
\begin{tikzpicture}[scale=0.8]
\draw [thick] (0,0) circle (2cm);
\filldraw (0,-2) circle (0.08cm);
\node at (0, -2.4) {$B$};
\filldraw (0,2) circle (0.08cm);
\node at (0, 2.4) {$A$};
\draw [thick, <-] (-2.2, 1) to[out= 60, in = 210] (-1.2, 2);
\draw [thick, <-] (-2.2, -1) to[out= 300, in = 150] (-1.2, -2);
\draw [thick] (6,0) circle (2cm);
\filldraw (6,-2) circle (0.08cm);
\node at (6, -2.4) {$B$};
\filldraw (6,2) circle (0.08cm);
\node at (6, 2.4) {$A$};
\filldraw (4,0) circle (0.08cm);
\node at (3.6, 0) {$C$};
\end{tikzpicture}
\caption{An antipodal configuration. Moving certain individual points from $A$ and $B$ to a common point $C$.}
\label{pic:type1}
\end{figure}
\end{center}

Once we are in this situation, an antipodal configuration where there is not a single point in $A$ or $B$ for which moving
it to $C$ further decreases the energy, we check whether there are any points which have the property that moving them to $C$ 
does not change the energy. If there are no such points, then we can set $C = \emptyset$. It is not too difficult to see that
we have then found a valid decomposition: since we cannot move any single point, we see that this implies
for all $v \in A$
$$ d_B(v) \geq d_A(v) + 1$$
and, likewise, for all $v \in B$ that $d_A(v) \geq d_B(v) + 1$.  If $C = \emptyset$, then the proof is complete at this stage. It remains to deal with the case where at least one point can be moved
to $C$ without changing the energy.

\subsection{Properties (1) and (2).}
If there exists a vertex in $A$ or $B$ such that moving it towards $C$ does not change the energy, then we do so while such
points exist. We note that
the order in which we move the vertices might play a role (moving a certain vertex may make it impossible to later
move another vertex which could originally have been chosen).  We move vertices in some arbitrary order from both $A$ and $B$ until we longer can. 
We can then conclude that, for each vertex $a \in A$, that moving it towards $C$ increases the energy: this implies that 
$$ d_C(a) + d_B(a) - d_A(a) \geq 1$$
since this number is integer-valued and if it were $\leq 0$, then we could move $a$ to $C$ without increasing the energy. 
Let us
now consider moving such an element $a \in A$ away from $A$ and $C$ and towards $B$ along the other side of the arc. 
If this decreases the energy, then we have found a way of decreasing the energy and by moving the point to the antipodal
point of $C$, we can decrease the energy by at least 1. We then return to the first part of the proof and use the rotation
argument to produce another antipodal configuration. We can thus assume that moving an element $a \in A$ away from 
$A$ and $C$ and towards $B$ along the other side of the arc will either increase the energy or keep it the same: 
this implies
$$ - d_C(a) + d_B(a) - d_A(a) \geq 0.$$
Combining both inequalities, we see that
$$ d_B(a) - d_A(a) \geq \max\left\{ 1 - d_C(a), d_C(a) \right\} = \max\left\{1, d_C(a) \right\}.$$
The same argument can be carried out in $B$ (the situation
is completely symmetric).

\subsection{Property (3)}
We now argue that there are no edges between two vertices located in the point $C$. Note first, that 
$$ g\left( \frac{\pi}{2} \right) = 0 = g\left( \frac{3\pi}{2} \right).$$
This means that, in terms of energy, there is no interaction between points in $C$ and points in $A \cup B$ (independently of how many edges there are) because $g$ vanishes. Moreover, each point in $C$ has exactly the same number of neighbors in $A$ and $B$. If there are now two vertices $c_1, c_2 \in C$ that are connected by
an edge, then moving one of these vertices, say $c_2$, to the antipodal point of $C$ can only affect edges that 
run between $c_2$ and other points in $C$ (and there is at least one such edge). Then, however, we see
that the energy decreases by 2 for each such edge if we do this. Moreover, we can take $c_2$ and move it continuously
to the other half of the circle: since $c_2$ has the same number of neighbors in $A$ and $B$, they do not have any effect
and the only effect arises from the edges between $c_2$ and the other vertices in $C$ (and there exists at least one such edge).
 We do it, if possible, reach a lower energy level and then go back to the first part of the argument, create another
antipodal configuration with lower energy until eventually there are no more edges between any
two vertices in $C$.

\begin{center}
\begin{figure}[h!]
\begin{tikzpicture}[scale=1]
\draw [thick] (0,0) circle (2cm);
\filldraw (0,-2) circle (0.08cm);
\node at (0, -2.3) {$B$};
\filldraw (0,2) circle (0.08cm);
\node at (0, 2.3) {$A$};
\filldraw (-2,0) circle (0.08cm);
\node at (-3, 0) {$c_1, c_2 \in C$};
\draw [dashed, thick] (-2,0) -- (0,2) -- (0,-2) -- (-2,0);
\draw [thick] (6,0) circle (2cm);
\filldraw (6,-2) circle (0.08cm);
\node at (6, -2.3) {$B$};
\filldraw (6,2) circle (0.08cm);
\node at (6, 2.3) {$A$};
\filldraw (4,0) circle (0.08cm);
\node at (3.7, 0) {$C$};
\filldraw (8,0) circle (0.08cm);
\node at (8.3, 0) {$c_2$};
\draw [dashed] (8,0) -- (6,2) -- (6,-2) -- (8,0) -- (4,0) -- (6,2) -- (6,-2) -- (4,0);
\end{tikzpicture}
\caption{Reflecting a single point from $C$ that is connected to another point in $C$ decreases the energy.}
\label{pic:ref}
\end{figure}
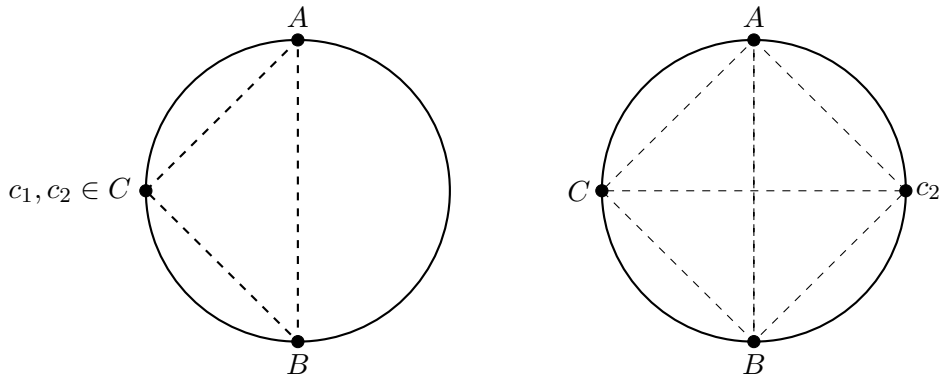
\end{center}

\subsection{Property (4).}
We may thus assume that there are no edges between any two vertices in $C$. We now see that by moving a single point in 
$C$ either up or down, we may be able to further decrease the energy. If this is the case, we have found another configuration
with yet smaller energy, we can decrease the energy by at least 1 and return to the first part of the argument to produce yet
another antipodal configuration with lower energy. We may thus assume that this is not possible: then, however, moving the points in $C$ up or down
leaves the energy invariant from which we can deduce that each point in $C$ has the same number of neighbors in $A$ and
$B$
$$ \forall c\in C: \qquad  d_A(c) = d_B(c).$$
This means that we can actually move the points in $C$ anywhere we want, it does not have any effect on the energy. This can also
be seen from the interpretation of the functional measuring the expected number of edges that are being cut by randomized
rounding: since each vertex in $C$ has the same number of neighbors in $A$ and $B$, it does not actually matter in which of
the two sets it ends up and thus, correspondingly, it does not affect the energy where on $\mathbb{S}^1$ it is located. Property (5)
has been shown above to be a consequence of Property (1) and Property (2) and the proof is complete.

\subsection{Remark.} We note that the proof of the Theorem was constructive. If the only goal would have been to show existence, then we could have assumed that the initial configuration $\theta \in (\mathbb{S}^1)^n$ is antipodal and corresponds to a solution of the \textsc{MaxCut} problem on the graph. This has the advantage that
$$  f(\theta_1, \dots, \theta_n) = 2\cdot |E| - 4 \cdot \textsc{MaxCut}(G)$$
is then already at the smallest possible value. In particular, it cannot decrease any further which simplifies the conceptual layout in the argument: there is never any further
descent direction. The rest of the argument remains the same. Naturally, this approach presupposes that one can find \textsc{MaxCut} which is NP-hard in general while our proof 
can be initialized with any arbitrary configuration and runs in polynomial time. In practical applications, it may be a advantageous to initialize with a \textsc{MaxCut} or approximate \textsc{MaxCut}
solution since it shortens the number of times that a smaller energy level can be encountered.

\section{Proof of Corollary 3}
We will try to bound the number of edges running between the two sets $A \cup C$ and $B$. We recall that
$$ \# E (A,C) = \# E(B,C),$$
so we are guaranteed to capture exactly half the edges of $E(C, A \cup B)$. It remains to study the behavior of the edges that run
between vertices of $A \cup B$. Let us call this restricted graph $H$ (see Fig. \ref{pic:rest} for a sketch). We first use the fact that each vertex in $A$ has more
neighbors in $B$ and vice versa to argue that
\begin{align*}
 |A| +|B| &\leq \sum_{a \in A} (d_B(a) - d_A(a)) + \sum_{b \in B} (d_A(b) - d_B(b)) \\
 &= 2\# E(A,B) - 2 \# E(A,A) - 2\# E(B,B).
 \end{align*}
Thus
$$   \# E(A,B)  \geq  \# E(A,A)  + \# E(B,B) +  \frac{|A| + |B|}{2}.$$
 \begin{center}
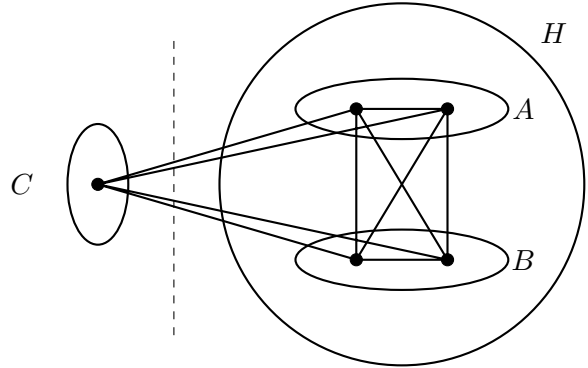
\begin{figure}[h!]
\begin{tikzpicture}[scale=2]
\draw [thick] (4,1) ellipse (0.7cm and 0.2cm);
\draw [thick] (4,0) ellipse (0.7cm and 0.2cm);
\draw [thick] (2, 0.5) ellipse (0.2cm and 0.4cm);
\filldraw (3.7,0) circle (0.04cm);
\filldraw (4.3,0) circle (0.04cm);
\filldraw (3.7,1) circle (0.04cm);
\filldraw (4.3,1) circle (0.04cm);
\filldraw (2,0.5) circle (0.04cm);
\draw[thick] (3.7,0) -- (4.3,0) -- (4.3,1) -- (3.7, 1) -- (2, 0.5) -- (3.7, 0) -- (3.7, 1) -- (4.3,0) -- (2, 0.5) -- (4.3, 1);
\draw[thick] (3.7,0) -- (4.3, 1);
\node at (4.8, 1) {$A$};
\node at (4.8, 0) {$B$};
\node at (1.5, 0.5) {$C$};
\draw [dashed] (2.5, -0.5) -- (2.5, 1.5);
\draw[thick] (4,0.5) circle (1.2cm);
\node at (5, 1.5) {$H$};
\end{tikzpicture}
\caption{$C$ contributes a precisely controlled amount, we ignore the subgraph on the vertices $A \cup B$ in isolation and call it $H$.}
\label{pic:rest}
\end{figure}
\end{center}
 This inequality is true independently of what happens in the set $C$ (recall that $C$ could be empty). However, using the refined formulation of Properties (1) and (2), 
 we see that as soon as $\#E(A,C) > |A|$, we necessarily have to have additional edges between $A$ and $B$. This allows
 us to improve the estimate by also incorporating this potential contribution
 $$   \# E(A,B)  \geq  \# E(A,A)  + \# E(B,B) +  \frac{|A| + |B|}{2} + \max\left\{0, \# E(A,C) - |A| \right\}.$$
By symmetry of the sets $A$ and $B$, we also have
 $$   \# E(A,B)  \geq  \# E(A,A)  + \# E(B,B) +  \frac{|A| + |B|}{2} + \max\left\{0, \# E(A,C) - |B| \right\}$$
from which we deduce
 \begin{align*}
    \# E(A,B)  &\geq  \# E(A,A)  + \# E(B,B) +  \frac{|A| + |B|}{2} \\
    &+ \max\left\{0, \# E(A,C) - \min\left\{|A|, |B| \right\} \right\}.
    \end{align*}

This is a lower bound on the number of edges that we can capture with the split $V = (A \cup C) \cup B$. It remains to understand
how this compares to the total number of edges.
Using the inequality just derived
\begin{align*}
 \# E(H) &= \# E(A,A)  + \# E(B,B) + \# E(A,B) \\
 &\leq 2 \#E(A,B) - \frac{|A| + |B|}{2} -  \max\left\{0, \# E(A,C) - \min\left\{|A|, |B| \right\} \right\}.
 \end{align*}
We now distinguish between two cases.\\

\textbf{Case 1. ($ \# E(A,C) \leq \min\left\{|A|, |B| \right\}$)}. In the first case, we have
$$  \# E(H) \leq 2 \#E(A,B) - \frac{|A| + |B|}{2}$$
and therefore, since $\# E (A,C) = \# E(B,C)$,
\begin{align*}
 |E| &= \# E(H) + \# E(A \cup B, c) \\
 &\leq 2 \#E(A,B) - \frac{|A| + |B|}{2} + 2\#E(B,C)\\
 &=  2 \#E(A \cup C,B) - \frac{|A| + |B|}{2}.
 \end{align*}
From this we can deduce that
$$ \frac{ \# E(A \cup C, B)}{|E|} \geq \frac{ \# E(A \cup C, B)}{2 \# E(A \cup C, B) - \frac{|A| + |B|}{2}}.$$

\textbf{Case 2. ($ \# E(A,C) \geq \min\left\{|A|, |B| \right\}$)}. In that case,  we have
\begin{align*}
 \# E(H) \leq 2 \#E(A,B) - \frac{|A| + |B|}{2} -  \# E(A,C) + \min\left\{|A|, |B| \right\}.
 \end{align*}
Since $\# E(A,C) = \# E(B,C)$, the total number of edges in the graph $G$ satisfies
\begin{align*}
|E| &\leq 2 \#E(A,B) - \frac{|A| + |B|}{2} + \# E(A,C) + \min\left\{|A|, |B| \right\}.
 \end{align*}
 Using the inequality defining Case 2 and $\# E(A,C) = \# E(B,C)$, we have
\begin{align*}
|E| &\leq 2 \#E(A,B) - \frac{|A| + |B|}{2} + 2\# E(A,C) \\
&= 2 \# E(A \cup C, B)- \frac{|A| + |B|}{2}.
 \end{align*}
This means that the ratio satisfies the same inequality as above. 
\begin{align*}
\frac{ \# E(A \cup C, B)}{|E|} &\geq \frac{ \# E(A \cup C, B)}{2 \# E(A \cup C, B) - \frac{|A| + |B|}{2}}.
 \end{align*}

\textbf{Conclusion.} At this point, we may use the inequality
\begin{align*}
\frac{ \# E(A \cup C, B)}{|E|} &\geq \frac{ \# E(A \cup C, B)}{2 \# E(A \cup C, B) - \frac{|A| + |B|}{2}}.
 \end{align*}
 unconditionally. The form of the inequality already suggests that we will be able to achieve a ratio
 larger than $1/2$.
Using the hand-shake lemma and the fact that $\# E(A,C) = \# E(B,C)$ together with $d(v) \leq \Delta$, we get
\begin{align*}
 \# E(A,B) &= \left(\frac{1}{2} \sum_{v \in A \cup B} d(v)\right) - \# E(A,A) - \# E(B,B)  - \# E(A,C)\\
 &\leq \Delta \frac{|A| + |B|}{2} -   \# E(A,A) - \# E(B,B) - \# E(A,C).
 \end{align*}
Using $\# E(A,C) = \# E(B,C)$  in the from $\# E(A \cup B, C) = 2\# E(A,C) $, we get
 \begin{align*}
  \Delta \frac{|A| + |B|}{2} &\geq    \# E(A,A) + \# E(B,B) + \# E(A,C) +   \# E(A,B)\\
  &=  \# E(A,A) + \# E(B,B) + \# E(A \cup B,C) +   \# E(A,B) - \# E(A,C)\\
  &= |E| - \# E(A,C).
  \end{align*}
and thus
\begin{align*}
\frac{|A| + |B|}{2}  \geq  \frac{1}{\Delta} (|E| - \# E(A,C)).
\end{align*}
 Property (5) implies $\# E(A \cup B, C) \leq 2 \#E(A, B)$ and thus
 $$ |E| \geq \# E(A \cup B, C) + \# E(A,B) \geq  \frac{3}{2}\# E(A \cup B, C)  $$
and therefore, using again $\# E(A,C) = \# E(B,C)$,
$$ \# E(A,C) = \frac{1}{2} \# E(A \cup B, C) \leq \frac{1}{3} |E|$$
from which we deduce
$$ \frac{|A| + |B|}{2}  \geq  \frac{2}{3\Delta} |E|.$$
Let us now suppose that, for some $0 < c < 1$,
$$ c = \frac{ \# E(A \cup C, B)}{|E|} \geq \frac{ \# E(A \cup C, B)}{2 \# E(A \cup C, B) - \frac{|A| + |B|}{2}}.$$
Then
\begin{align*}
 c &\geq  \frac{ \# E(A \cup C, B)}{2 \# E(A \cup C, B) - \frac{2}{3 \Delta} |E|} \\
 &=   \frac{  c \cdot |E|}{2 c |E| - \frac{2}{3 \Delta} |E|}  = \frac{1}{2 - \frac{2}{3\Delta} \frac{1}{c}}
\end{align*}
which forces
$$ c \geq \frac{1}{2} + \frac{1}{3\Delta}.$$
This concludes the argument.

\end{document}